\documentclass[letterpaper,10pt]{article}

\usepackage{amsthm}
\usepackage{amsmath}
\usepackage{amssymb}
\usepackage{subfig}
\usepackage{authblk}
\usepackage{geometry}
\usepackage{palatino}
\usepackage{graphicx}
\usepackage{lineno}
\usepackage{setspace}
\usepackage{color}
\usepackage{scalerel,stackengine}
\usepackage[]{algorithm2e}
\usepackage{cleveref}

\stackMath
\newcommand\reallywidehat[1]{%
\savestack{\tmpbox}{\stretchto{%
  \scaleto{%
    \scalerel*[\widthof{\ensuremath{#1}}]{\kern-.6pt\bigwedge\kern-.6pt}%
    {\rule[-\textheight/2]{1ex}{\textheight}}
  }{\textheight}%
}{0.5ex}}%
\stackon[1pt]{#1}{\tmpbox}%
}

\DeclareMathOperator*{\argmin}{argmin}
\DeclareMathOperator*{\sgn}{\mbox{sgn}}

\newtheorem{theorem}{Theorem}

\newtheorem{corollary}{Corollary}

\newtheorem{remark}{\textit{Remark}}

\title{$l_p$ regularization for ensemble Kalman inversion}
\author[1]{Yoonsang Lee\thanks{yoonsang.lee@dartmouth.edu}}
\affil[1]{Department of Mathematics, Dartmouth College}

\begin{document}

\maketitle

\begin{abstract}
Ensemble Kalman inversion (EKI) is a derivative-free optimization method that lies between the deterministic and the probabilistic approaches for inverse problems. EKI iterates the Kalman update of ensemble-based Kalman filters, whose ensemble converges to a minimizer of an objective function. EKI regularizes ill-posed problems by restricting the ensemble to the linear span of the initial ensemble, or by iterating regularization with early stopping. Another regularization approach for EKI, Tikhonov EKI, penalizes the objective function using the $l_2$ penalty term, preventing overfitting in the standard EKI.
This paper proposes a strategy to implement $l_p, 0<p\leq 1,$ regularization for EKI to recover sparse structures in the solution. The strategy transforms a $l_p$ problem into a $l_2$ problem, which is then solved by Tikhonov EKI.
The transformation is explicit, and thus the proposed approach has a computational cost comparable to Tikhonov EKI. We validate the proposed approach's effectiveness and robustness through a suite of numerical experiments, including compressive sensing and subsurface flow inverse problems.
\end{abstract}

\section{Introduction}
A wide range of problems in science and engineering are formulated as inverse problems.
Inverse problems aim to estimate a quantity of interest from noisy, imperfect observation or measurement data, such as state variables or a set of parameters that constitute a forward model.
Examples include deblurring and denoising in image processing \cite{CS},  recovery of permeability in subsurface flow using pressure fields \cite{OIL}, and training a neural network in machine learning \cite{NN,MLEKI} to name a few. In this paper, we consider the inverse problem of finding $u\in\mathbb{R}^N$ from measurement data $y\in\mathbb{R}^m$ where $u$ and $y$ are related as follows
\begin{equation}\label{eq:inverse}
y=G(u)+\eta.
\end{equation}
Here $G:\mathbb{R}^N\to\mathbb{R}^m$ is a forward model that can be nonlinear and computationally expensive to solve, for example, solving a PDE problem. The last term $\eta$ is a measurement error. The measurement error is unknown in general, but we assume that it is drawn from a known probability distribution, a Gaussian distribution with mean zero and a known covariance $\Gamma$. By assuming that the forward model $G$ and the observation covariance $\Gamma$ are known, the unknown variable $u$ is estimated by solving an optimization problem
\begin{equation}\label{eq:optimization}
\argmin_{u\in \mathbb{R}^N} \frac{1}{2}\|y-G(u)\|^2_{\Gamma},
\end{equation}
where $\|\cdot\|_{\Gamma}$ is the norm induced from the inner product using the inverse of the covariance matrix $\Gamma$, that is $\|a\|_{\Gamma}^2=\langle a,\Gamma^{-1}a\rangle$ for the standard inner product $\langle,\rangle$ in $\mathbb{R}^m$.

Ensemble Kalman inversion (EKI), pioneered in the oil industry \cite{OIL} and mathematically formulated in an application-neutral setting in \cite{EKI}, is a derivative-free method that lies between the deterministic and the probabilistic approaches for inverse problems.
EKI's key feature is an iterative application of the Kalman update of the ensemble-based Kalman filters \cite{EAKF, EnKF}.
Ensemble-based Kalman filters are well known for their success in numerical weather prediction, stringent inverse problems involving high-dimensional systems.
EKI iterates the ensemble-based Kalman update in which the ensemble mean converges to the solution of the optimization problem \cref{eq:optimization}.
EKI can be thought of as a least-squares method in which the derivatives are approximated from an empirical correlation of an ensemble \cite{statisticalderivative}, not from a variational approach. Thus, EKI is highly parallelizable without calculating the derivatives related to the forward or the adjoint problem used in the gradient-based methods.

Inverse problems are often ill-posed, which suffer from non-uniqueness of the solution and lack stability.
Also, in the context of regression, the solution can show overfitting.
A common strategy to overcome ill-posed problems is regularizing the solution of the optimization problem \cite{regularization}.
That is, a special structure of the solution from prior information, such as sparsity, is imposed to address ill-posedness.
The standard EKI \cite{EKI} implements regularization by restricting the ensemble to the linear span of the initial ensemble reflecting prior information.
The ensemble-based Kalman update is known for that the ensemble remains in the linear span of the initial ensemble \cite{span,EKI}.
Thus, the EKI ensemble always stays in the linear span of the initial ensemble, which regularizes the solution.
Although this approach shows robust results in certain applications, numerical evidence demonstrates that overfitting may still occur \cite{EKI}.
As an effort to address the overfitting of the standard EKI, an iterative regularization method has been proposed in \cite{iterativeregularization}, which approximates the regularizing Levenberg-Marquardt scheme \cite{LM}.
As another regularization approach using a penalty term to the objective function, a recent work called Tikhonov EKI (TEKI) \cite{TEKI} implements the Tikhonov regularization (which imposes a $l_2$ penalty term to the objective function) using an augmented measurement model that adds artificial measurements to the original measurement.
TEKI's implementation is a straightforward modification of the standard EKI method with a marginal increase in the computational cost.

The regularization methods for EKI mentioned above address several issues of ill-posed problems, including overfitting. However, it is still an open problem to implement other types of regularizers, such as $l_1$ or total variation (TV) regularization.
This paper aims to implement $l_p, 0<p\leq 1$, regularization to recover sparse structures in the solution of inverse problems.
In other words, we propose a highly-parallelizable derivative-free method that solves the following $l_p$ regularized optimization problem
\begin{equation}\label{eq:lpinu}
\argmin_{u\in X}\frac{\lambda}{2}\|u\|_p^p+\frac{1}{2}\|y-G(u)\|^2_{\Gamma},
\end{equation}
where $\|u\|_p$ is the $l_p$ norm of u, i.e., $\sum_i^N |u_i|^p$, and $\lambda$ is a regularization coefficient.
The proposed method's key idea is a transformation of variables that converts the $l_p$ regularization problem to the Tikhonov regularization problem. Therefore, a local minimizer of the original $l_p$ problem can be found by a local minimizer of the $l_2$ problem that is solved using the idea of Tikhonov EKI.
As this transformation is explicit and easy to calculate, the proposed method's overall computational complexity remains comparable to the complexity of Tikhonov EKI.
In general, a transformed optimization problem can lead to additional difficulties, such as change of convexity, increased nonlinearity, additional/missing local minima of the original problem, etc. \cite{practicaloptimization}.
We show that the transformation does not add or remove local minimizers in the transformed formulation. A work imposing sparsity in EKI has been reported recently \cite{sparseEKI}. The idea of this work is to use thresholding and a $l_1$ constraint to impose sparsity in the inverse problem solution. The $l_1$ constraint is further relaxed by splitting the solution into positive and negative parts. The split converts the $l_1$ problem to a quadratic problem, while it still has a non-negativity constraint. On the other hand, our method does not require additional constraints by reformulating the optimization problem and works as a solver for the $l_p$ regularized optimization problem \cref{eq:lpinu}.

This paper is structured as follows. \Cref{sec:EKI}  reviews the standard EKI and Tikhonov EKI.
In \cref{sec:lpEKI}, we describe a transformation that converts the $l_p$ regularization problem \cref{eq:lpinu}, $0<p\leq 1$, to the Tikhonov (that is, $l_2$) regularization problem, and provide the complete description of the $l_p$ regularized EKI algorithm. We also discuss implementation and computation issues.
\Cref{sec:tests} is devoted to the validation of the effectiveness and robustness of regularized EKI through a suite of numerical tests. The tests include a scalar toy problem with an analytic solution, a compressive sensing problem to benchmark with a convex $l_1$ minimization method, and a PDE-constrained nonlinear inverse problem from subsurface flow. We conclude this paper in \cref{sec:conclusion}, discussing the proposed method's limitations and future work.

\section{Ensemble Kalman inversion}\label{sec:EKI}
The $l_p$ regularized EKI uses a change of variables to transform a $l_p$ problem into a $l_2$ problem, which is then solved by the standard EKI using an augmented measurement model.
This section reviews the standard EKI and the application of the augmented measurement model in Tikhonov EKI to implement $l_2$ regularization. The review is intended to be concise, delivering the minimal ideas for the $l_p$ regularized EKI. Detailed descriptions of the standard EKI and the Tikhonov EKI methods can be found in \cite{EKI} and \cite{TEKI}, respectively.

\subsection{Standard ensemble Kalman inversion}\label{subsec:EKI}
EKI incorporates an artificial dynamics, which corresponds to the application of the forward model to each ensemble member. This application moves each ensemble member to the measurement space, which is then updated using the ensemble Kalman update formula. 
The ensemble updated by EKI stays in the linear span of the initial ensemble \cite{EKI, span}. Therefore, by choosing an initial ensemble appropriately for prior information, EKI is regularized as the ensemble is restricted to the linear span of the initial ensemble.
Under a continuous-time limit, when the operator $G$ is linear, it is proved in \cite{EKIanalysis} that EKI estimate converges to the solution of the following optimization problem
\begin{equation}
\argmin_{u\in\mathbb{R}^N}\frac{1}{2}\|y-G(u)\|_{\Gamma}^2.
\end{equation}
In this paper, we consider the discrete-time EKI in \cite{EKI}, which is described below.


\textbf{Algorithm: standard EKI}\\
Assumption: an initial ensemble of size $K$, $\{u_0^{(k)}\}_{k=1}^K$ from prior information, is given.\\
For $n=1,2,...,$
\begin{enumerate}
	\item Prediction step using the artificial dynamics:
	\begin{enumerate}
		\item Apply the forward model $G$ to each ensemble member 
\begin{equation}
g_n^{(k)}:=G(u_{n-1}^{(k)})
\end{equation}
		\item From the set of the predictions $\{g_n^{(k)}\}_{k=1}^K$, calculate the mean and covariances
\begin{equation}\label{eq:samplemean}
\overline{g}_n=\frac{1}{K}\sum_{k=1}^Kg_n^{(k)},
\end{equation}
\begin{equation}\label{eq:samplecovariance}
\begin{split}
C^{ug}_n&=\frac{1}{K}\sum_{k=1}^K(u_n^{(k)}-\overline{u}_n)\otimes(g_n^{(k)}-\overline{g}_n),\\
C^{gg}_n&=\frac{1}{K}\sum_{k=1}^K(g_n^{(k)}-\overline{g}_n)\otimes(g_n^{(k)}-\overline{g}_n),
\end{split}
\end{equation}
where $\overline{u}_n$ is the mean of $\{u_n^{(k)}\}$, i.e., $\displaystyle\frac{1}{K}\sum_{k=1}^Ku_n^{(k)}$.
	\end{enumerate}

\item Analysis step:
	\begin{enumerate}
		\item Update each ensemble member $u_n^{(k)}$ using the Kalman update
\begin{equation}\label{eq:ensembleupdate}
u_{n+1}^{(k)}=u_{n}^{(k)}+C^{ug}_n(C^{gg}_n+\Gamma)^{-1}(y_{n}^{(k)}-g_n^{(k)}),
\end{equation}
where $y_{n+1}^{(k)}=y+\zeta_{n+1}^{(k)}$ is a perturbed measurement using Gaussian noise $\zeta_{n+1}^{(k)}$ with mean zero and covariance $\Gamma$.

		\item Compute the mean of the ensemble as an estimate for the solution
		\begin{equation}
		\overline{u}_{n+1}=\frac{1}{K}\sum_{k=1}^Ku_n^{(k)}
		\end{equation}
	\end{enumerate}
\end{enumerate}

\begin{remark}
The term $C^{ug}_n(C^{gg}_n+\Gamma)^{-1}$ in \cref{eq:ensembleupdate} is from the Kalman gain matrix. The standard EKI uses an extended space, $(u,G(u))\in\mathbb{R}^{N+m}$, and then use the Kalman update for the extended space variable. However, as we need to update only $u$ while $G(u)$ is subordinate to $u$, we have the update formula \cref{eq:ensembleupdate}.
\end{remark}

\subsection{Tikhonov ensemble Kalman inversion}\label{subsec:TKEI}
EKI is regularized through the initial ensemble reflecting prior information. However, there are several numerical evidence showing that EKI regularized only through an ensemble may have overfitting \cite{EKI}.
Among other approaches to regularize EKI, Tikhonov EKI \cite{TEKI} uses the idea of an augmented measurement to implement $l_2$ regularization, which is a simple modification of the standard EKI.
For the original measurement $y$, the augmented measurement model extends $y$ by adding the zero vector in $\mathbb{R}^N$, which yields an augmented measurement vector $z\in\mathbb{R}^{m+N}$
\begin{equation}\label{eq:augmentedmeasurement}
\mbox{augmented measurement vector: }z=(y,0).
\end{equation}
The forward model is also augmented to account for the augmented measurement vector, which adds the identity measurement
\begin{equation}\label{eq:augmentedforward}
\mbox{augmented forward model: }F(u)=(G(u), u).
\end{equation}
Using the augmented measurement vector and the model, Tikhonov EKI has the following inverse problem of estimating $u$ from $z$
\begin{equation}
z=F(u)+\zeta.
\end{equation}
Here $\zeta$ is a $m+N$-dimensional measurement error for the augmented measurement model, which is Gaussian with mean zero and covariance
\begin{equation}\label{eq:augmentedcovariance}
\Sigma=\begin{pmatrix}\Gamma&0\\0&\frac{1}{\lambda}I_{N}\end{pmatrix},
\end{equation}
for the $N\times N$ identity matrix $I_N$.

The mechanism enabling the $l_2$ regularization in Tikhonov EKI is the incorporation of the $l_2$ penalty term as a part of the augmented measurement model. From the orthogonality between different components in $\mathbb{R}^{m+N}$, we have
\begin{equation}
\begin{split}
\frac{1}{2}\|z-F(u)\|^2_{\Sigma}&=\frac{1}{2}\|y-G(u)\|^2_{\Gamma}+\frac{1}{2}\|0-u\|^2_{\frac{1}{\lambda}I_N}\\
&=\frac{1}{2}\|y-G(u)\|^2_{\Gamma}+\frac{\lambda}{2}\|u\|^2_2.\\
\end{split}
\end{equation}
Therefore, the standard EKI algorithm applied to the augmented measurement minimizes $\frac{1}{2}\|z-F(u)\|^2_{\Sigma}$, which equivalently minimizes the $l_2$ regularized problem.

\section{$l_p$-regularization for EKI}\label{sec:lpEKI}
This section describes a transformation that converts a $l_p, 0<p\leq 1,$ regularization problem to a $l_2$ regularization problem. $l_p$-regularized EKI ($l_p$EKI), which we completely describe in \cref{subsec:lpEKI}, utilizes this transformation and solves the transformed $l_2$ regularization problem using the idea of Tikhonov EKI \cite{TEKI}, the augmented measurement model.

\subsection{Transformation of $l_p$ regularization into $l_2$ regularization}\label{subsec:transformation}
For $0<p\leq 1$, we define a function $\psi:\mathbb{R}\to\mathbb{R}$ given by
\begin{equation}\label{eq:utovcomponent}
\psi(x)=\sgn(x)|x|^{\frac{p}{2}}, \quad x\in\mathbb{R}.
\end{equation}
Here $\sgn(x)$ is the sign function of $x$, which has 1 for $x>0$, 0 for $x=0$, and -1 for $x<0$. It is straightforward to check that $\psi$ is bijective and has an inverse $\xi:\mathbb{R}\to\mathbb{R}$ defined as
\begin{equation}\label{eq:vtoucomponent}
\xi(x)=\sgn(x)|x|^{\frac{2}{p}}, \quad x\in\mathbb{R}.
\end{equation}
For $u$ in $\mathbb{R}^N$, we define a nonlinear map $\Psi:\mathbb{R}^N\to\mathbb{R}^N$, which applies $\psi$ to each component of $u=(u_1,u_2,...,u_N)$,\begin{equation}\label{eq:utov}
\Psi(u)=(\psi(u_1),\psi(u_2),...,\psi(u_N)).
\end{equation}
As $\psi$ has an inverse, the map $\Psi$ also has an inverse, say $\Xi$
\begin{equation}\label{eq:vtou}
\Xi(u)=\Psi^{-1}(u)=(\xi(u_1),\xi(u_2),...,\xi(u_N)).
\end{equation}
For $v=\Psi(u)$, it can be checked that for each $i=1,2,...,N$,
\[|v_i|^2=|\psi(u_i)|^2=|u_i|^{p},\]
and thus we have the following norm relation
\begin{equation}\label{eq:normequivalence}
\|v\|_2^2=\|u\|_p^p.
\end{equation}
This relation shows that the map $v=\Psi(u)$ converts the $l_p$-regularized optimization problem in $u$ \cref{eq:lpinu} to a $l_2$ regularized problem in $v$,
\begin{equation}\label{eq:l2inv}
\argmin_{v\in \mathbb{R}^N}\frac{\lambda}{2}\|v\|_2^2+\frac{1}{2}\|y-\tilde{G}(v)\|^2_{\Gamma},
\end{equation}
where $\tilde{G}$ is the pullback of $G$ by $\Xi$
\begin{equation}
\tilde{G}=G\circ \Xi.
\end{equation}

A transformation between $l_1$ and $l_2$ regularization terms has already been used to solve an inverse problem in the Bayesian framework \cite{l1RTO}. In the context of the randomize-then-optimize framework \cite{RTO}, the method in \cite{l1RTO} draws a sample from a Gaussian distribution, which is then transformed to a Laplace distribution. As this method needs to match the corresponding densities of the variables (the original and the transformed variables) as random variables, the transformation involves calculations related to cumulative distribution functions. For the scalar case, $v\in\mathbb{R}$, the transformation from $l_2$ to $l_1$, denoted as $gl$, is given by
\begin{equation}\label{eq:l1rto}
gl(v)=-\sgn(v)\log\left(1-2\left|\phi(v)-\frac{1}{2}\right|\right).
\end{equation}
where $\phi(u)$ is the cumulative distribution function of the standard Gaussian distribution.
\cref{fig:transformations} shows the two transformations $\xi$ \cref{eq:vtoucomponent} and $gl$ \cref{eq:l1rto}; the former is based on the norm relation \cref{eq:normequivalence} and the latter is based on matching densities as random variables.
\begin{figure}[!ht]
\centering
\includegraphics[width=.5\textwidth]{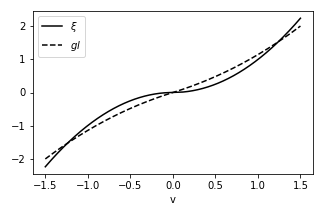}
\caption{$\xi$: transformation matching the norm relation \cref{eq:normequivalence}, $gl$: transformation from Gaussian to Laplace distributions.}\label{fig:transformations}
\end{figure}
We note that the transformation $\xi$ has a region around 0 flatter than the transformation $gl$, but $\xi$ diverts quickly as $v$ moves further away from $0$. From this comparison, we expect that the flattened region of $\xi$ plays another role in imposing sparsity by trapping the ensemble to the flattened area.


In general, a reformulation of an optimization problem using a transformation has the following potential issues \cite{practicaloptimization}: i) the degree of nonlinearity may be significantly increased, ii) the desired minimum may be inadvertently excluded, or iii) an additional local minimum can be included.
In \cite{EKIconvergence}, for a non-convex problem, it is shown that TEKI converges to an approximate local minimum if the gradient and Hessian of the objective function are bounded. It is straightforward to check that the transformed objective function has bounded gradient and Hessian if $0<p\leq 1$ regardless of the convexity of the problem. Therefore, if we can show that the original and the transformed problems have the same number of local minima, then it is guaranteed to find a local minimum of the original problem by finding a local minimum of the transformed problem using TEKI. We want to note the importance of the sign function in defining $\psi$ and $\xi$. The sign function is not necessary to satisfy the norm relation \cref{eq:normequivalence}, but it is essential to make the transformation $\Psi$ and its inverse $\Xi$ bijective. Without being bijective, the transformed $l_2$ problem can have more or less local minima than the original problem.

The following theorem shows that the transformation does not add or remove local minima.
\begin{theorem} For an objective function $J(u):\mathbb{R}^N\to\mathbb{R}$, if $u^*$ is a local minimizer of $J(u)$, $\Psi(u^*)$ is also a local minimizer of $\tilde{J}(v)=J\circ\Xi(v)$. Similarly, if $v^*$ is a local minimizer of $\tilde{J}(v)$, then $\Xi(v^*)$ is also a local minimizer of $J(u)=\tilde{J}\circ \Psi(u)$.
\end{theorem}
\begin{proof}
From the definition \cref{eq:utov} and \cref{eq:vtou}, $\Psi$ and $\Xi$ are continuous and bijective. Thus for $u\in\mathbb{R}^N$, both $\Psi$ and $\Xi$ map a neighborhood of $u\in\mathbb{R}^N$ to neighborhoods of $\Psi(u)$ and $\Xi(u)$, respectively. As $u^*$ is a local minimizer, there exists a neighborhood $\mathcal{N}$ of $u^*$ such that
\begin{equation}
J(u^*)\leq J(w)\quad \mbox{for all }w\in\mathcal{N}.
\end{equation}
Let $v=\Psi(u^*)$ and $\mathcal{M}:=\Psi(\mathcal{N})$ that is a neighborhood of $v$. For any $w\in\mathcal{M}$, $\Xi(w)\in\mathcal{N}$ and thus we have
\begin{equation}
\tilde{J}(v)=J(\Xi(v))=J(u)\leq J(\Xi(w))=\tilde{J}(w),
\end{equation}
which shows that $v$ is a local minimizer of $\tilde{J}$. The other direction is proved similarly by changing the roles of $\Psi$ and $\Xi$ and of $J$ and $\tilde{J}$.
\end{proof}
We note that an insolated local minimizer can replace the local minimizer in the theorem.
If there is a unique global minimizer of the $l_p$ regularization problem \cref{eq:lpinu}, the theorem guarantees that we can find it by finding the global minimizer of the $l_2$ regularized problem \cref{eq:l2inv}.
\begin{corollary}
For $0<p\leq 1$, if the $l_p$ regularized optimization \cref{eq:lpinu} has a unique global minimizer, say $u^{\dag}$, the $l_2$ regularized optimization \cref{eq:l2inv} also has a unique global minimizer. By finding the minimizer $u^{\dag}$ of \cref{eq:l2inv}, say $v^{\dag}$, $u^{\dag}$ is given by
\begin{equation}u^{\dag}=\Xi(v^{\dag}).
\end{equation}
\end{corollary}

\subsection{Algorithm}\label{subsec:lpEKI}
$l_p$-regularized EKI ($l_p$EKI) solves the transformed $l_2$ regularization problem using the standard EKI with the augmented measurement model. For the current study's completeness to implement $l_p$EKI, this subsection describes the complete $l_p$EKI algorithm and discuss issues related to implementation. Note that the Tikhonov EKI (TEKI) part in $l_p$EKI is slightly modified to reflect the setting assumed in this paper. The general TEKI algorithm and its variants can be found in \cite{TEKI}.

We assume that the forward model $G$ and the measurement error covariance $\Gamma$ are known, and measurement $y\in\mathbb{R}^m$ is given (and thus $z=(y,0)$ is also given). We also fix the regularization coefficient $\lambda$ and $p$. Under this assumption, $l_p$EKI uses the following iterative procedure to update the ensemble until the ensemble mean $\overline{v}=\displaystyle\frac{1}{K}\sum_{k=1}^Kv^{(k)}$ converges.

\vspace{0.02\textwidth}
\textbf{Algorithm:  $l_p$-regularized EKI}\\
Assumption: an initial ensemble of size $K$, $\{v_0^{(k)}\}_{k=1}^K$, is given.\\
For $n=1,2,...,$
\begin{enumerate}
	\item Prediction step using the forward model:
	\begin{enumerate}
		\item Apply the augmented forward model $F$ to each ensemble member
\begin{equation}
f_n^{(k)}:=F(v_n^{(k)})=(\tilde{G}(v_n^{(k)}), v_n^{(k)})
\end{equation}
		\item From the set of the predictions $\{f_n^{(k)}\}_{k=1}^K$, calculate the mean and covariances
\begin{equation}\label{eq:rEKI:samplemean}
\overline{f}_n=\frac{1}{K}\sum_{k=1}^Kf_n^{(k)},
\end{equation}
\begin{equation}\label{eq:rEKI:samplecovariance}
\begin{split}
C^{vf}_n&=\frac{1}{K}\sum_{k=1}^K(v_n^{(k)}-\overline{v}_n)\otimes(f_n^{(k)}-\overline{f}_n),\\
C^{ff}_n&=\frac{1}{K}\sum_{k=1}^K(f_n^{(k)}-\overline{f}_n)\otimes(f_n^{(k)}-\overline{f}_n)
\end{split}
\end{equation}
where $\overline{v}_n$ is the ensemble mean of $\{v_n^{(k)}\}$, i.e., $\displaystyle\frac{1}{K}\sum_{k=1}^Kv_n^{(k)}$.
	\end{enumerate}

\item Analysis step:
	\begin{enumerate}
		\item Update each ensemble member $v_n^{(k)}$ using the Kalman update
\begin{equation}\label{eq:rEKI:ensembleupdate}
v_{n+1}^{(k)}=v_{n}^{(k)}+C^{vf}_n(C^{ff}_n+\Sigma)^{-1}(z_{n+1}^{(k)}-f_n^{(k)}),
\end{equation}
where $z_{n+1}^{(k)}=z+\zeta_{n+1}^{(k)}$ is a perturbed measurement using Gaussian noise $\zeta_{n+1}^{(k)}$ with mean zero and covariance $\Sigma$.

		\item For the ensemble mean $\overline{v}_n$, the $l_p$EKI estimate, $u_n$, for the minimizer of the $l_p$ regularization is given by
	\begin{equation}\label{eq:rEKI:finalestimate}
	u = \Xi(\overline{v}_n).
	\end{equation}
	\end{enumerate}
\end{enumerate}

\begin{remark} In EKI and TEKI, the covariance of $\zeta_{n+1}^{(k)}$ can be set to zero so that all ensemble member uses the same measurement $z$ without perturbations. In our study, we focus on the perturbed measurement using the covariance matrix $\Gamma$.
\end{remark}

\begin{remark} The above algorithm is equivalent to TEKI, except that the forward model $G$ is replaced with the pullback of $G$ by the transformation $\Xi$. In comparison with TEKI, the additional computational cost for $l_p$EKI is to calculate the Transformation $\Xi(v)$. In comparison with the standard EKI, the additional cost of $l_p$EKI, in addition to the cost related to the transformation, is the matrix inversion $(C^{gg}_n+\Sigma)^{-1}$ in the augmented measurement space $\mathbb{R}^{m+N}$ instead of a matrix inversion in the original measurement space $\mathbb{R}^m$. As the covariance matrices are symmetric positive definite, the matrix inversion can be done efficiently.
\end{remark}

\begin{remark} In $l_p$EKI, it is also possible to consider estimating $u$ by transforming each ensemble member and take average of the transformed members, that is,
\begin{equation}\label{eq:rEKI:finalestimate2}
u=\frac{1}{K}\sum_{k=1}^K \Xi(v_n^{(k)})
\end{equation}
instead of \cref{eq:rEKI:finalestimate}. If the ensemble spread is large, these two approaches will make a difference. In our numerical tests in the next section, we do not incorporate covariance inflation. Thus the ensemble spread becomes relatively small when the estimate converges, and thus \cref{eq:rEKI:finalestimate} and \cref{eq:rEKI:finalestimate2} are not significantly different. In this study, we use \cref{eq:rEKI:finalestimate} to measure the performance of $l_p$EKI.
\end{remark}
 
In recovering sparsity using the $l_p$ penalty term, if the penalty term's convexity is not necessary, it is preferred to use a small $p<1$ as a smaller $p$ imposes stronger sparsity. The transformation in $l_p$EKI works for any positive $p$, but the transformation can lead to an overflow for a small $p$; the function $\xi$ depends on an exponent $\frac{2}{p}$ that becomes large for a small $p$. Therefore, there is a limit for the smallest $p$. In our numerical experiments in the next section, the smallest $p$ is 0.7 in the compressive sensing test.

There is a variant of $l_p$EKI worth further consideration. In \cite{EKIanalysis}, a continuous-time limit of EKI has been proposed, which rescales $\Gamma\to h^{-1}\Gamma$ using $h>0$ so that the matrix inversion $(C^{gg}_n+h^{-1}\Gamma)^{-1}$ is approximated by $h\Gamma^{-1}$ as a limit of $h\to 0$. In many applications, the measurement error covariance is assumed to be diagonal. That is, the measurement error corresponding to different components are uncorrelated. Thus the inversion $\Gamma^{-1}$ becomes a cheap calculation in the continuous-time limit. The continuous-time limit is then discretized in time using an explicit time integration method with a finite time step. The latter is called `learning rate' in the machine learning community, and it is known that an adaptive time-stepping to solve an optimization often shows improved results \cite{adaptivelearningrate,adaptivestepping}. The current study focuses on the discrete-time update described in \cref{eq:ensembleupdate} and we leave adaptive time-stepping for future work.

\section{Numerical tests}\label{sec:tests}
We apply $l_p$-regularized EKI  ($l_p$EKI) to a suite of inverse problems to check its performance in regularizing EKI and recovering sparse structures of solutions.
The tests include: i) a scalar toy model where an analytic solution is available, ii) a compressive sensing problem to recover a sparse signal from random measurements of the signal, iii) an inverse problem in subsurface flow; estimation of permeability from measurements of hydraulic pressure field whose forward model is described by a 2D elliptic partial differential equation \cite{Darcy, OIL}. In all tests, we run $l_p$EKI for various values of $p\leq 1$, and compare with the result of Tikhonov EKI. We analyze the results to check how effectively $l_p$EKI implements $l_p$ regularization and recover sparse solutions. When available, we also compare $l_p$EKI with a $l_1$ convex minimization method. As quantitative measures for the estimation performance, we calculate the $l_1$ error of the $l_p$EKI estimates and the data misfit $\|y-G(u)\|_{2}$.

Several parameters are to be determined in $l_p$EKI to achieve robust estimation results, the regularization coefficient $\lambda$, ensemble size, and its initialization. The regularization coefficient can be selected, for example, using cross-validation. As the coefficient can significantly affect the performance, we find the coefficient by hand-tuning so that $l_p$EKI achieves the best result for a given $p$. 
Ensemble initialization plays a role in regularizing EKI, restricting the estimate to the linear span of the initial ensemble.
In our experiments, instead of tuning the initial ensemble for improved results,
we initialize the ensemble using a Gaussian distribution with mean zero and a constant diagonal covariance matrix (the variance will be specified later for each test). As this initialization does not utilize any prior information, a sparse structure in the solution, we regularize the solution mainly through the $l_p$ penalty term.
For each test, we run 100 trials of $l_p$EKI through 100 realizations of the initial ensemble distribution and use the estimate averaged over the trials along with its standard deviation to measure the performance difference.

Regarding the ensemble size, for the scalar toy and the compressive sensing problems, we test ensemble sizes larger than the dimension of $u$, the unknown variable of interest. The purpose of a large ensemble size is to minimize the sampling error while we focus on the regularization effect of $l_p$EKI. To show the applicability of $l_p$EKI for high-dimensional problems, we also test a small ensemble size using the idea of multiple batches used in \cite{batch}. The multiple batch approach runs several batches where small magnitude components are removed after each batch. After removing small magnitude components from the previous batch, the ensemble is used for the next batch. The multiple batch approach enables a small ensemble size, 50 ensemble members, for the compressive sensing and the 2D elliptic inversion problems where the dimensions of $u$ are 200 and 400, respectively.

In ensemble-based Kalman filters, covariance inflation is an essential tool to stabilize and improve the performance of the filters. In a connection with the inflation, an adaptive time-stepping has been investigated to improve the performance of EKI. Although the adaptive time-stepping can be incorporated in $l_p$EKI for performance improvements, we use the discrete version $l_p$EKI described in \cref{subsec:lpEKI} focusing on the effect of different types of regularization on inversion. We will report a thorough investigation along the line of adaptive time-stepping in another place.

%
%
\subsection{A scalar toy problem}\label{subsec:scalar}
The first numerical test is a scalar problem for $u\in\mathbb{R}$ with an analytic solution.
As this is a scalar problem, there is no effect of regularization from ensemble initialization, and we can see the effect from the $l_p$ penalty term. The scalar optimization problem we consider here is the minimization of an objective function $J(u)=\frac{1}{4}|u|^p+\frac{1}{2}(1-u)^2$
\begin{equation}\label{eq:scalar}
\argmin_{u\in\mathbb{R}}J(u)=\argmin_{u\in\mathbb{R}}\frac{1}{4}|u|^p+\frac{1}{2}(1-u)^2.
\end{equation}
This setup is equivalent to solving the optimization problem \cref{eq:lpinu} using $l_p$ regularization with $\lambda=1/2$, where $y=1$, $G(u)=u$, and $\eta$ is Gaussian with mean zero and variance 1. Using the transformation $v=\Psi(u)=\psi(u)=\sgn(u)|u|^{\frac{p}{2}}$ defined in \cref{eq:utovcomponent}, $l_p$EKI minimizes a transformed objective function $\tilde{J}(v)=\frac{1}{4}|v|^2+\frac{1}{2}(1-\sgn(v)|v|^{2/p})^2$
\begin{equation}\label{eq:scalartransformed}
\argmin_{v\in\mathbb{R}}\tilde{J}(v)=\argmin_{v\in\mathbb{R}}\frac{1}{4}|v|^2+\frac{1}{2}(1-\sgn(v)|v|^{2/p})^2,
\end{equation}
which is an $l_2$ regularization of $\frac{1}{2}(1-\sgn(v)|v|^{\frac{2}{p}})^2$.

\begin{figure}[!ht]
\centering
\includegraphics[width=.98\textwidth]{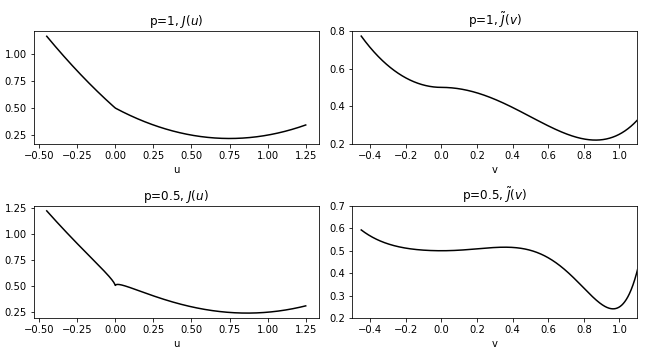}
\caption{Objective functions of \cref{eq:scalar} and \cref{eq:scalartransformed} for $p=1$ (first row) and $p=0.5$ (second row).}\label{fig:scalarobjective}
\end{figure}
For $p=1$, the first row of \cref{fig:scalarobjective} shows the objective functions of $l_p$ \cref{eq:scalar} and the transformed $l_2$ \cref{eq:scalartransformed} formulations. Each objective function has a unique global minimum without other local minima. The minimizers are $\frac{3}{4}$ and $\frac{\sqrt{3}}{2}$ for $l_1$ and $l_2$, respectively. We can check that the transformation does not add/remove local minimizers, but the convexity of the objective function changes. The transformed objective function $\tilde{J}$ has an inflection point at $u=0$, which is also a stationary point. Note that the original function has no other stationary points than the global minimizer.

When $p=0.5$, a potential issue of the transformation can be seen explicitly. The original objective and the transformed objective functions are shown in the second row of \cref{fig:scalarobjective}. Due to the regularization term with $p=0.5$, the objective functions are non-convex and have a local minimizer at $u=v=0$ in addition to the global minimizers. In the transformed formulation (bottom right of \cref{fig:scalarobjective}), the objective function flattens around $v=0$, which shows a potential issue of trapping ensemble members around $v=0$. Numerical experiments show that if the ensemble is initialized with a small variance, the ensemble is trapped around $v=0$.
On the other hand, if the ensemble is initialized with a sufficiently large variance (so that some of the ensemble members are initialized out of the well around $v=0$), $l_p$EKI shows convergence to the true minimizer, $v=0.9304$ (or $u=0.8656$) even when it is initialized around $0$.

\begin{figure}[!ht]
\centering
\includegraphics[width=0.95\textwidth]{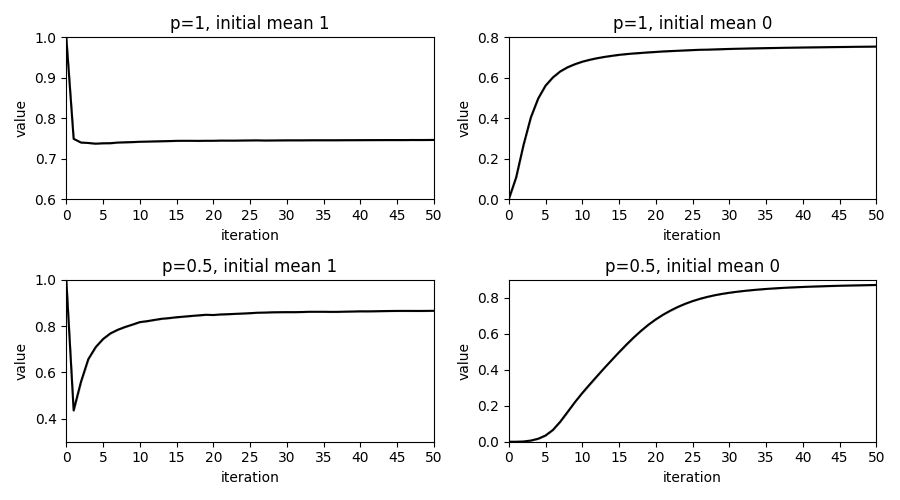}
\caption{Time series of $l_p$EKI estimate, $\xi(\overline{v}_n)$, which is averaged over 100 different trials. }
\label{fig:scalartimeseries}
\end{figure}
We use 100 different realizations for the ensemble initialization and each trial uses 50 ensemble members. The estimates at each iteration, which is averaged over different trials, are shown in \cref{fig:scalartimeseries}.
For $p=1$ (first row) and $p=0.5$ (second row), the left and right columns show the results when the ensemble is initialized with mean 1 and 0, respectively. When $p=1$ and initialized around $1$, the ensemble estimate quickly converges to the true value $0.75$ as the objective function is convex, and the initial guess is close to the true value. When $p=0.5$, as the objective function is non-convex due to the regularization term, the convergence is slower than the $p=1$ case. When the ensemble is initialized around 0 for $p=0.5$, a local minimizer, the ensemble needs to be initialized with a large variance. Using variance 1, which is 10 times larger than 0.1, the variance for the ensemble initialization around 1, $l_p$EKI converges to the true value. The performance difference between different trials is marginal. The standard deviations of the estimate after 50 iterations are $6.62\times 10^{-3}$ ($p=1$ initialized with 1), $7.95\times 10^{-3}$ ($p=1$ initialized with 0), $8.79\times 10^{-3}$ ($p=0.5$ initialized with 1), and $1.14\times 10^{-2}$ ($p=0.5$ initialized with 0). As a reference, the estimate using the transformation \cref{eq:l1rto} based on matching the densities of random variables converges to 0.71.


%
%
\subsection{Compressive sensing}
The second test is a compressive sensing problem. The true signal $u$ is a vector in $\mathbb{R}^{200}$, which is sparse with only four randomly selected non-zero components (their magnitudes are also randomly chosen from the standard normal distribution.)
The forward model $G:\mathbb{R}^{200}\to\mathbb{R}^{20}$ is a random Gaussian matrix of size $20\times 200$, which yields a measurement vector in $\mathbb{R}^{20}$. The measurement $y$ is obtained by applying the forward model to the true signal $u$ polluted by Gaussian noise with mean zero and variance $0.01$. As the forward model is linear, several robust methods can solve the sparse recovery problem, including the $l_1$ convex minimization method \cite{cvx}. This test aims to compare the performance of $l_p$EKI for various $p$ values, rather than to advocate the use of $l_p$EKI over other standard methods. As the forward model is linear and cheap to calculate, the standard methods are preferred over $l_p$EKI for this test.

We first use a large ensemble size, 2000 ensemble members, to run $l_p$EKI. The ensemble is initialized by drawing samples from a Gaussian distribution with mean zero and a diagonal covariance (which yields variance 0.1 for each component).
For $p=1$ and $0.7$, the tuned regularization coefficients, $\lambda$, are 100 and 300. When $p=2$, which corresponds to TEKI, the best result can be obtained using $\lambda$ ranging from 10 to 200; we use the result of $\lambda=50$ to compare with the other cases. For $p=1$, we also compare the result of the convex $l_1$ minimization method using the KKT solver in the Python library CVXOPT \cite{cvxopt}.
%

\begin{figure}[!ht]
\centering
\includegraphics[width=.98\textwidth]{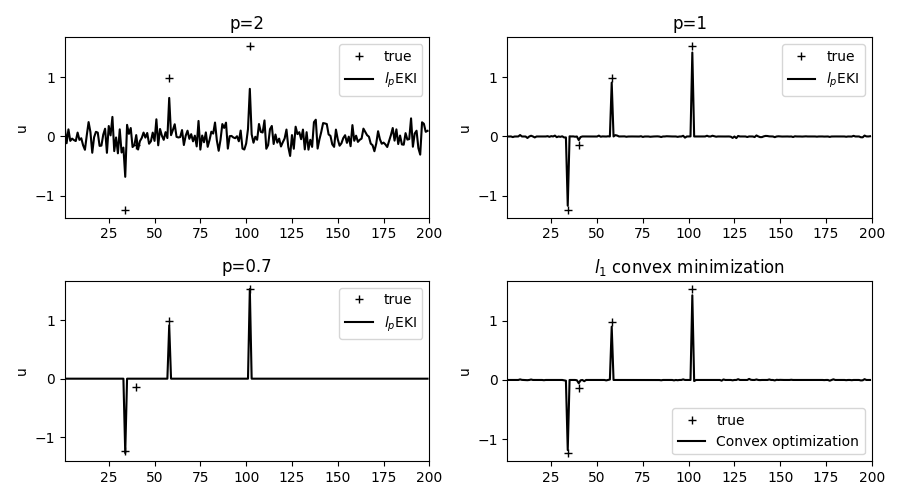}
\caption{Compressive sensing. Reconstruction of sparse signal using $l_p$EKI for $p$=2, 1, and 0.7. Ensemble size is 2000. The bottom right plot is the reconstruction using the convex $l_1$ minimization method. For the true signal, only the nonzero components are marked}.\label{fig:csestimate}
\end{figure}

\Cref{fig:csestimate} shows the $l_p$EKI estimates after 20 iterations averaged over 100 trials for $p=2$ (top left), $p=1$ (top right), and $p=0.7$ (bottom left), along with the estimate by the convex optimization (bottom right). 
As it is well known in compressive sensing, $l_2$ regularization fails to capture the true signal's sparse structure. As $p$ decreases to 1, $l_p$EKI develops sparsity in the estimate, comparable to the estimate of the convex $l_1$ minimization method. The slightly weak magnitudes of the three most significant components by $l_p$EKI improve as $p$ decreases to $0.7$. When $p=0.7$, $l_p$EKI captures the correct magnitudes at the cost of losing the smallest magnitude component. We note that the smallest magnitude component is difficult to capture; the magnitude is comparable to the measurement error $0.1=\sqrt{0.01}$.

\begin{table}[t!]
\centering
\begin{tabular}{|c|c|c|}
\hline
Method&$l_1$ error&data misfit\\
\hline
$p=2$, ens size 2000&14.0802&0.0515\\
$p=1$, ens size 2000&0.7848&0.8018\\
$p=0.7$, ens size 2000&0.2773&1.2737\\
$p=1$, ens size 50&1.6408&1.4095\\
$p=0.7$, ens size 50&0.6027&1.8958\\
$l_1$ convex minimization&0.5623&0.9030\\
\hline
\end{tabular}
\caption{Compressive sensing. $l_p$EKI estimate $l_1$ error and data misfit for $p=2,1$ and $0.7$. }\label{tb:csmisfitl1err}
\end{table}

\begin{figure}[!ht]
\centering
\includegraphics[width=.98\textwidth]{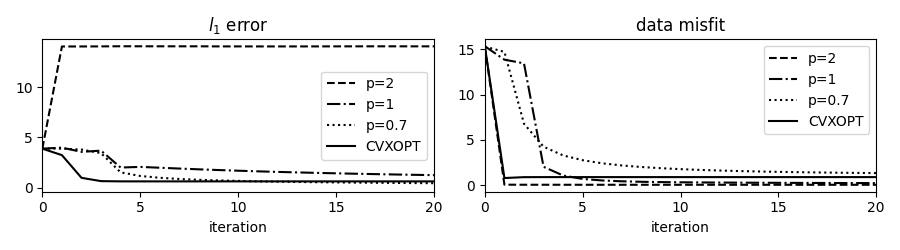}
\caption{Compressive sensing. $l_1$ error of the $l_p$EKI estimate and data misfit.}\label{fig:csmisfitl1err}
\end{figure}
Another cost of using $p<1$ to impose stronger sparsity than $p=1$ is a slow convergence rate of$l_p$EKI. The time series of the $l_1$ estimation error and the data misfit of $l_p$EKI averaged over 100 trials are shown in \cref{fig:csmisfitl1err} alongside those of the convex optimization method. The results show that $p=0.7$ converges slower than $p=1$ (see Table \cref{tb:csmisfitl1err} for the numerical values of the error and the misfit). Although there is a slowdown in convergence, it is worth noting that $l_p$EKI with $p=0.7$ converges in a reasonably short time, 15 iterations, to achieve the best result. $l_p$EKI with $p=2$ converges fast with the smallest data misfit. However, the $l_2$ regularization is not strong enough to impose sparsity in the estimate and yields the largest estimation error, which is 20 times larger than the case of $p=1$. Note that the convex optimization method has the fastest convergence rate; it converges within three iterations and captures the four nonzero components with slightly smaller magnitudes than $p=0.7$ for the three most significant components.

\begin{figure}[!ht]
\centering
\includegraphics[width=.98\textwidth]{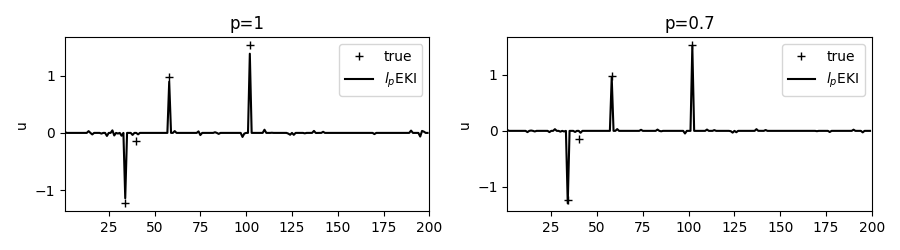}
\caption{Compressive sensing. Reconstruction of sparse signal using $l_p$EKI for p=1 and 0.7. Ensemble size is 50. For the true signal, only the nonzero components are marked.}\label{fig:csestimate_nsamp50}
\end{figure}
The ensemble size 2000 is larger than the dimension of the unknown vector $u$, 200. A large ensemble size can be impractical for a high-dimensional unknown vector. To see the applicability of $l_p$EKI using a small ensemble size, we use 50 ensemble members and two batches following the multiple batch approach \cite{batch}. The first batch runs 10 iterations, and all components whose magnitudes are less than 0.1 (the square root of the observation variance) are removed. The problem's size the second batch solves ranges from 30-45 (depending on a realization of the initial ensemble), which is then solved for another 10 iterations. The estimates using 50 ensemble members for $p=1$ and $p=0.7$ after two batch runs (i.e., 20 iterations) are shown in \cref{fig:csestimate_nsamp50}. Compared with the large ensemble size case, the small ensemble size run also captures the most significant components at the cost of fluctuating components larger than the large ensemble size test. We note that the estimates are averaged over 100 trials, and thus there are components whose magnitudes are less than the threshold value 0.1 used in the multiple batch run.

\begin{figure}[!ht]
\centering
\includegraphics[width=.98\textwidth]{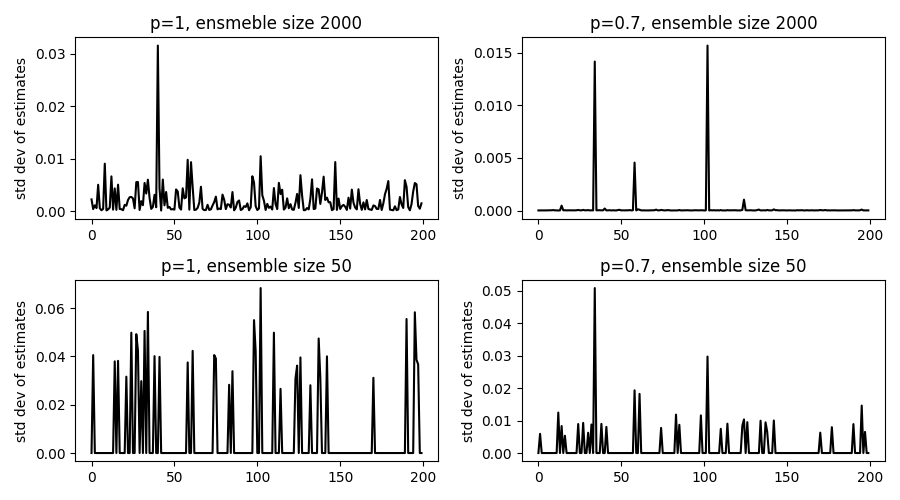}
\caption{Compressive sensing. Standard deviation of the estimates using 100 trials.}\label{fig:csstdrest}
\end{figure}

As a measure to check the performance difference for different trials, \cref{fig:csstdrest} shows the standard deviations of $l_p$EKI estimates for $p=1$ and $0.7$ after 20 iterations. The first row shows the results using 2000 ensemble members, while the second row shows the ones using 50 ensemble members. 
The standard deviations of the large ensemble size are smaller than those of the small ensemble size case as the large ensemble size has a smaller sampling error. In all cases, the standard deviations are smaller than 6\% of the magnitude of the most significant components. In terms of $p$, the standard deviations of $p=0.7$ are smaller than those of $p=1$.

\subsection{2D elliptic problem}
Next, we consider an inverse problem where the forward model is given by an elliptic partial differential equation.
The model is related to subsurface flow described by Darcy flow in the two-dimensional unit square $(0,1)^2\subset\mathbb{R}^2$
\begin{equation}\label{eq:2d}
-\nabla \cdot (k(x)\nabla p(x)) = f(x), \quad x=(x_1,x_2)\in(0,1)^2.
\end{equation}
The scalar field $k(x)>\alpha>0$ is the permeability, and another field $p(x)$ is the piezometric head or the pressure field of the flow. For a known source term $f(x)$, the inverse problem estimates the permeability from measurements of the pressure field $p$. This model is a standard model for an inverse problem in oil reservoir simulations and has been actively used to measure EKI's performance and its variants, including TEKI \cite{EKI, TEKI}.

We follow the same setting used in TEKI \cite{TEKI} for the boundary conditions and the source term.
The boundary conditions consist of Dirichlet and Neumann boundary conditions
\[p(x_1,0)=100,  \frac{\partial p}{\partial x_1}(1,x_2)=0, -k\frac{\partial p}{\partial x_1}(0,x_2)=500,  \frac{\partial p}{\partial x_2}(x_1,1)=0,\]
and the source term is piecewise constant
\[\displaystyle f(x_1,x_2)=\left\{
\begin{array}{ll}
0&\mbox{if } 0\leq x_2\leq \frac{4}{6},\\
137&\mbox{if }\frac{4}{6}< x_2\leq \frac{5}{6},\\
274&\mbox{if }\frac{5}{6}<x_2\leq 1.
\end{array}\right.\]
A physical motivation of the above configuration can be found in \cite{Darcy}.
We use $15\times 15$ regularly spaced points in $(0,1)^2$ to measure the pressure field with a small measurement error variance $10^{{-6}}$. For a given $k$, the forward model is solved by a FEM method using the second-order polynomial basis on a $60\times 60$ uniform mesh.

In addition to the standard setup, we impose a sparse structure in the permeability. We assume that the
log permeability, $\log k$, can be represented by 400 components in the cosine basis $\phi_{ij}=\cos(i\pi x_1)\cos(j\pi x_2), i,j=0,1,...,19,$
\begin{equation}\log k(x)=\sum_{i,j=0}^{19}u_{ij}\phi_{ij}(x),\end{equation}
where only six of $\{u_{ij}\}$ are nonzero. That is, we assume that the discrete cosine transform of $\log k$ is sparse with only 6 nonzero components out of 400 components.
Thus, the problem we consider here can be formulated as an inverse problem to recover $u=\{u_{ij}\}\in\mathbb{R}^{400}$ (which has only six nonzero components) from a measurement $y\in\mathbb{R}^{225}$, the measurement of $p$ at $15\times 15$ regularly spaced points.
In terms of sparsity reconstruction, the current setup is similar to the previous compressive sensing problem, but the main difference lies in the forward model.
In this test, the forward model is nonlinear and computationally expensive to solve, where the forward model in the compressive sensing test was linear using a random measurement matrix.

For this test, we run $l_p$EKI using only a small ensemble size due to the high computational cost of running the forward model. As in the previous test, we use the multiple batch approach. First, the $l_p$EKI ensemble of size 50 is initialized around zero with Gaussian perturbations of variance 0.1. After the first five iterations, all components whose magnitudes less than $5\times 10^{-3}$ are removed at each iteration. The threshold value is slightly smaller than the smallest magnitude component of the true signal. Over 100 different trials, the average number of nonzero components after 30 iterations is 18 that is smaller than the ensemble size.

\begin{figure}[!htp]
\centering
\subfloat[true]{\includegraphics[width=.98\textwidth]{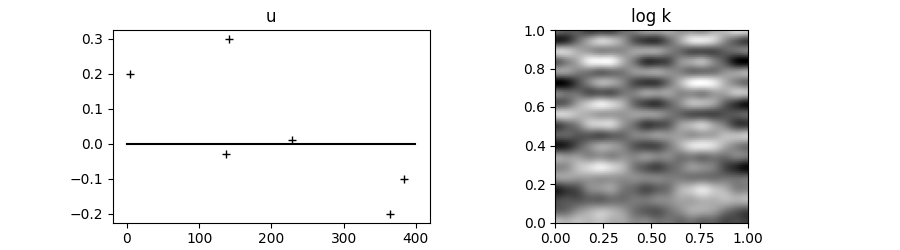}}\\
\subfloat[$p=2$]{\includegraphics[width=.98\textwidth]{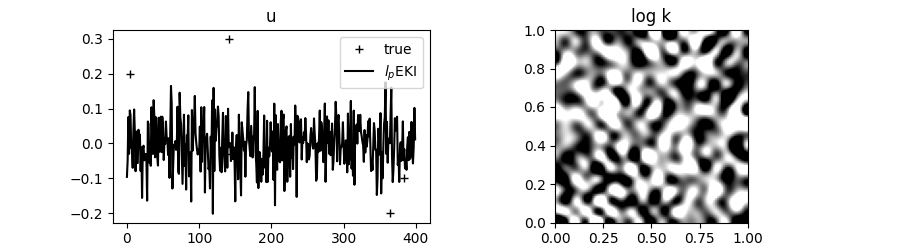}}\\
\subfloat[$p=1$]{\includegraphics[width=.98\textwidth]{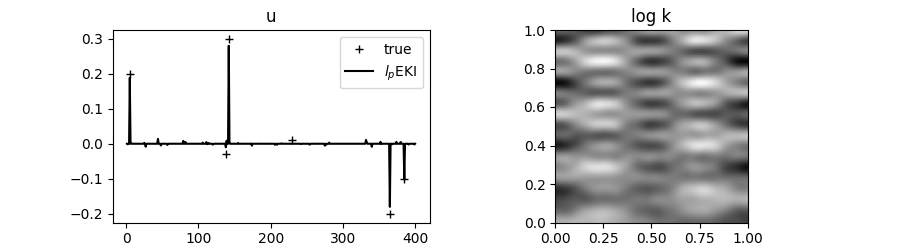}}\\
\subfloat[$p=0.8$]{\includegraphics[width=.98\textwidth]{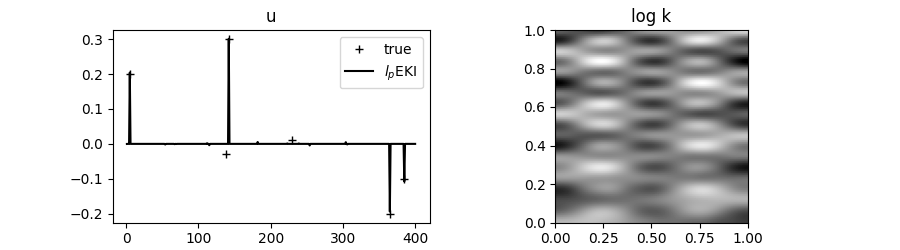}}
\caption{2D elliptic problem. Left column: the true $u$ and $l_p$EKI  estimates for $p=2, 1,$ and $0.8$. Right column: $\log k$ of the true and $l_p$EKI estimates. All plots have the same grey scale. $p=1$ and $0.8$ use the results after 20 iterations while $p=2$ uses the result after 50 iterations. For the true signal, only the nonzero components are marked.}\label{fig:2destimates}
\end{figure}

The true value of $u$ used in this test and its corresponding log permeability, $\log k$, are shown in the first row of \cref{fig:2destimates} ($u$ is represented as a one-dimensional vector by concatenating the row vectors of $\{u_{ij}\}$).
The $l_p$EKI estimates for $p=2,1$, and $0.8$ are shown in the second to the fourth rows of \cref{fig:2destimates}. Here $p=0.8$ was the smallest value we can use for $l_p$EKI due to the numerical overflow in the exponentiation of $\log k$. A smaller $p$ can be used with a smaller variance for ensemble initialization, but the gain is marginal.
The results of $l_p$EKI are similar to the compressive sensing case.
$p=0.8$ has the best performance recovering the four most significant components of $u$.
$p=1$ has slightly weak magnitudes missing the correct magnitudes of the two most significant components (corresponding to one-dimensional indices 141 and 364).
Both cases converge within 20 iterations to yield the best result (see \cref{fig:2dmisfitl1err} and Table \cref{tb:2dmisfitl1err} for the time series and numerical values of the $l_1$ error and data misfit). When $p=2$, $l_p$EKI performs the worst; it has the largest $l_1$ error and data misfit. We note that $p=2$ uses the result after running 50 iterations at which the estimate converges.

\begin{figure}[t!]
\centering
\includegraphics[width=.98\textwidth]{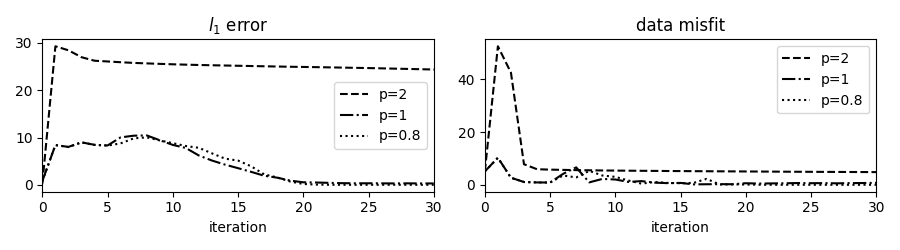}
\caption{2D elliptic problem. $l_1$ error of the $l_p$EKI estimates and data misfit.}\label{fig:2dmisfitl1err}
\end{figure}
\begin{figure}[htp]
\centering
\includegraphics[width=.98\textwidth]{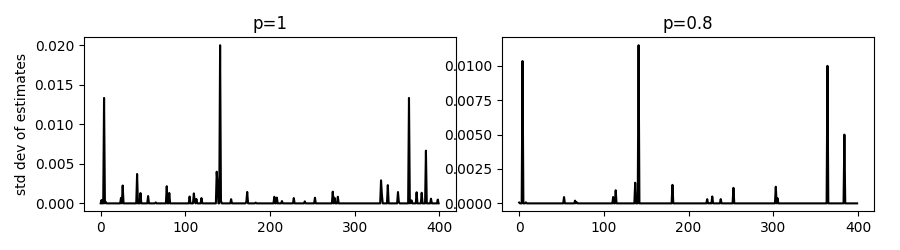}
\caption{2D elliptic problem. Standard deviation of the estimates using 100 trials.}\label{fig:2dstd}
\end{figure}
The performance difference between different trials is not significant. The standard deviations of the $l_p$EKI estimates using 100 trials are shown in \cref{fig:2dstd}. The standard deviations for nonzero components are larger than the other components, but the largest standard deviation is less than 3\% of the magnitude of the true signal. As in the compressive sensing test, the deviations are slightly smaller for $p<1$ than $p=1$.

\begin{table}[t!]
\centering
\begin{tabular}{|c|c|c|}
\hline
$p$&$l_1$ error&data misfit\\
\hline
2&21.3389&4.1227\\
1&0.1553&0.5707\\
0.8&0.0719&0.5682\\
\hline
\end{tabular}
\caption{2D elliptic problem. $l_p$EKI estimate $l_1$ error and data misfit for $p=2,1$ and $0.8$. }\label{tb:2dmisfitl1err}
\end{table}

\section{Discussions and conclusions}\label{sec:conclusion}
We have proposed a strategy to implement $l_p, 0<p\leq 1$, regularization in ensemble Kalman inversion (EKI) to recover sparse structures in the solution of an inverse problem.
The $l_p$-regularized ensemble Kalman inversion ($l_p$EKI) proposed here uses a transformation to convert the $l_p$ regularization to the $l_2$ regularization, which is then solved by the standard EKI with an augmented measurement model used in Tikhonov EKI. We showed a one-to-one correspondence between the local minima of the original and the transformed formulations. Thus a local minimum of the original problem can be obtained by finding a local minimum of the transformed problem.
As other iterative methods for non-convex problems, initialization plays a vital role in the proposed method's performance.
The effectiveness and robustness of regularized EKI are validated through a suite of numerical tests, showing robust results in recovering sparse solutions using $p\leq1$.

In implementing $l_p$ regularization for EKI, there is a limit for $p<1$ due to an overflow. One possible workaround is to use a nonlinear augmented measurement model related to the transformation $\Psi$, not the transformation $\Xi$. The nonlinear measurement model is general to incorporate the $l_p$ regularization term directly instead of using the transformed $l_2$ problem. However, this approach lacks a mathematical framework to prevent the inadvertent addition of local minima. This approach is under investigation and will be reported in another place.

 For successful applications of $l_p$EKI for high-dimensional inverse problems, it is essential to maintain a small ensemble size for efficiency. In the current study, we considered the multiple batch approach. The approach removes non-significant components after each batch, and thus the problem size (i.e., the dimension of the unknown signal) decreases over different batch runs. This approach enabled $l_p$EKI to use only 50 ensemble members to solve 200 and 400-dimensional inverse problems. Other techniques, such as variance inflation and localization, improve the performance of the standard EKI using a small ensemble size \cite{EKIanalysis}. It would be natural to investigate if these techniques can be extended to $l_p$EKI to decrease the sampling error of $l_p$EKI.  
 
In the current study, we have left several variants of $l_p$EKI for future work. Weighted $l_1$ has been shown to recover sparse solutions using fewer measurements than the standard $l_1$ \cite{weightedl1}. It is straightforward to implement weighted $l_1$ (and further weighted $l_p$ for $p<1$) in $l_p$EKI by replacing the identity matrix in \cref{eq:augmentedcovariance} with another type of covariance matrix corresponding to the desired weights. We plan to study several weighting strategies to improve the performance of$l_p$EKI. As another variant of$l_p$EKI, we plan to investigate the adaptive time-stepping under the continuous limit. The time step for solving the continuous limit equation, which is called `learning rate' in the machine learning community, is known to affect an optimization solver \cite{adaptivelearningrate}. The standard ensemble Kaman inversion has been applied to machine learning tasks, such as discovering the vector fields defining a differential equation, using time series data \cite{MLEKI} and sparse learning using thresholding \cite{sparseEKI}. We plan to investigate the effect of an adaptive time-stepping for performance improvements and compare with the sparsity EKI method using thresholding in dimension reduction in machine learning.


\bibliographystyle{siamplain}
\bibliography{regularizedEKI}
\end{document}